\documentclass{article}
\usepackage{amsmath,amssymb,amsthm}
\usepackage{bm}
\newtheorem{theorem}{Theorem}[section]
\newtheorem{proposition}[theorem]{Proposition}
\newtheorem{lemma}[theorem]{Lemma}

\theoremstyle{definition}
\newtheorem{definition}[theorem]{Definition}
\theoremstyle{remark}

\newcommand{\N}{\mathbb{N}}

\newcommand{\Q}{\mathbb{Q}}
\newcommand{\R}{\mathbb{R}}
\newcommand{\C}{\mathbb{C}}

\DeclareMathOperator{\re}{Re}

\newcommand{\balpha}{\bm{\alpha}}
\newcommand{\bbeta}{\bm{\beta}}

\newcommand{\bmu}{\bm{\mu}}

\newcommand{\alphavec}{\alpha_1,\ldots,\alpha_s}
\newcommand{\betavec}{\beta_1,\ldots,\beta_t}
\newcommand{\muvec}{\mu_1,\ldots,\mu_p}

\newcommand{\calS}{\mathcal{S}}

\begin{document}
\title{Multiple series expressions for the Newton series which
 interpolate finite multiple harmonic sums}
\author{Gaku Kawashima}
\date{}
\maketitle

\begin{abstract}
 The Newton series which interpolate finite multiple harmonic sums
 are useful in the study of multiple zeta values (MZV's).
 In this paper, we prove that these Newton series can be written as
 multiple series.
 As an application, we give a formula for MZV's
 which contains the duality.
\end{abstract}

\section{Introduction} \label{sec0}
Let $k_1,\ldots,k_p$ be positive integers and $k_1 \ge 2$.
Then the following nested series
\begin{displaymath}
 \zeta(k_1,\ldots,k_p) = 
 \sum_{n_1 > \cdots > n_p > 0}
 \frac{1}{n_1^{k_1} \cdots n_p^{k_p}}
\end{displaymath}
is called a multiple zeta value,
which is introduced by Hoffman \cite{H} and
Zagier \cite{Z}.
It is known that there are many $\Q$-linear relations among these numbers.
For example, the simplest one is the relation
\begin{displaymath}
 \zeta(3) = \zeta(2,1),
\end{displaymath}
which is due to Euler.
In recent years, a great deal of work has been done
on the problem of finding all such relations.
Among others, the regularized double shuffle relations
studied by Ihara, Kaneko and Zagier are conjectured
to imply all $\Q$-linear relations \cite{IKZ}.
\par
The Newton series which interpolate finite multiple harmonic sums
seem to be useful for the study of relations 
among multiple zeta values (MZV's).
In fact, in \cite{K}, a formula which contains Ohno's relation
was derived by using these series.
(Ohno's relation is a generalization of the duality and the sum formula
and gives many relations among MZV's \cite{O}.)
Let $\balpha$ be a multi-index (i.e. an ordered set of positive
integers).
Then the Newton series used in \cite{K} is as follows:
\begin{displaymath}
 F_{\balpha}(z) = 
 \sum_{n=0}^{\infty} (-1)^n (\nabla S_{\balpha})(n) \binom{z}{n},
\end{displaymath}
where 
$\nabla$ denotes the inversion operator on the space
$\C^{\N}$ of complex-valued sequence (see Definition \ref{df1-20})
and the sequence $S_{\balpha}$ is defined by
\begin{displaymath}
 S_{\balpha}(n) =
 \sum_{n > n_1 \ge \cdots \ge n_p \ge 0}
 \frac{1}{(n_1+1)^{\alpha_1} \cdots (n_p+1)^{\alpha_p}},
 \quad n \in \N.
\end{displaymath}
(In this paper, we denote by $\N$ the set of all non-negative integers.)
The series $F_{\balpha}(z)$ converges in the half-plane $\re z > -1$
and takes the value $S_{\balpha}(n)$ at $z=n \in \N$.
In addition, its Taylor coefficients at the origin are expressed 
in terms of MZV's.
Our objective in this paper is to present another expression
of the Newton series $F_{\balpha}(z)$.
This Newton series can be written as a multiple series.
\par
In Section \ref{sec2}, we will define a multiple series $G_{\balpha}(z)$
for any multi-index $\balpha$.
For example, we have
\begin{align*}
 G_{3,3}(z) = \sum_{n_1 > n_2 > n_3 \ge n_4 > n_5 > n_6 > 0}
 &\left(\frac{1}{n_1} - \frac{1}{n_1+z}\right)
 \frac{1}{n_2n_3(n_4+z)n_5n_6}, \\[-4mm]
 &\hspace{4pt}\underbrace{\hspace{9em}}_3 \:
 \underbrace{\hspace{5.2em}}_3
\end{align*}
\begin{align*}
 G_{2,2,2}(z) = \sum_{n_1 > n_2 \ge n_3 > n_4 \ge n_5 > n_6 > 0}
 &\left(\frac{1}{n_1} - \frac{1}{n_1+z}\right)
 \frac{1}{n_2(n_3+z)n_4(n_5+z)n_6} \\[-4mm]
 &\hspace{4pt}\underbrace{\hspace{7.9em}}_2 \:
 \underbrace{\hspace{4.2em}}_2 \:
 \underbrace{\hspace{4.2em}}_2 
\end{align*}
and
\begin{align*}
 G_{1,1,3,1}(z) = \sum_{n_1 \ge n_2 \ge n_3 > n_4 > n_5 \ge n_6 > 0} 
 &\left(\frac{1}{n_1} - \frac{1}{n_1+z}\right)
 \frac{1}{(n_2+z)(n_3+z)n_4n_5(n_6+z)}. \\[-4mm]
 &\hspace{4pt}\underbrace{\hspace{6.6em}}_1 \: 
 \underbrace{\hspace{3.3em}}_1 \:
 \underbrace{\hspace{5.3em}}_3 \:
 \underbrace{\hspace{3.3em}}_1
\end{align*}
The following is our main result:
For any multi-index $\balpha$, we have
\begin{displaymath}
 F_{\balpha}(z) = G_{\balpha^{\backprime}}(z), \quad \re z > -1,
\end{displaymath}
where $\balpha^{\backprime}$ is the dual multi-index of $\balpha$
and is defined in Section \ref{sec1}.
(Actually, we give the above equation in a larger half-plane.
See Theorem \ref{th3-40}.)
As an application of this equation, we prove a formula for MZV's
which contains the duality (Proposition \ref{prop4-30}).
\section{Multi-indices and finite multiple harmonic sums} \label{sec1}
In this section, we give some definitions related to multi-indices
and finite multiple harmonic sums which are used throughout the paper.
\par
A finite sequence of positive integers is called a
multi-index. 
The length and the weight of a multi-index $\balpha=(\alphavec)$ 
are defined to be
\begin{displaymath}
 l(\balpha) := s \quad \text{and} \quad 
 |\balpha| := \alpha_1 + \cdots + \alpha_s,
\end{displaymath}
respectively.
Let $m$ be a positive integer.
The multi-indices of weight $m$ are in one-to-one
correspondence with the subsets of the set $\{1,2,\ldots,m-1\}$ by
\begin{displaymath}
 \calS_m \colon (\alphavec) \mapsto 
 \{\alpha_1, \alpha_1 + \alpha_2, \ldots, 
 \alpha_1 + \alpha_2 + \cdots + \alpha_{s-1}\}.
\end{displaymath}
For example, in the case $m=3$, we have
\begin{displaymath}
 (3) \mapsto \emptyset, \quad (1,2) \mapsto \{1\}, \quad (2,1) \mapsto
 \{2\} \quad \text{and} \quad (1,1,1) \mapsto \{1,2\}
\end{displaymath}
from the diagrams
\begin{displaymath}
  {\arraycolsep=1pt
  \begin{array}{ccccc}
          & &        & &         \\
  \bigcirc& &\bigcirc& &\bigcirc \\
          &1&        &2&        
  \end{array}}\,,\qquad
  {\arraycolsep=1pt
  \begin{array}{ccccc}
          &\downarrow&        & &         \\
  \bigcirc&          &\bigcirc& &\bigcirc \\
          &1         &        &2&
  \end{array}}\,, \qquad
  {\arraycolsep=1pt
  \begin{array}{cccccccccc}
          & &        &\downarrow&         \\
  \bigcirc& &\bigcirc&          &\bigcirc \\
          &1&        &2         &
  \end{array}}\qquad \text{and} \qquad
  {\arraycolsep=1pt
  \begin{array}{ccccc}
          &\downarrow&        &\downarrow &         \\
  \bigcirc&          &\bigcirc&           &\bigcirc \\
          &1         &        &2          &
  \end{array}}\,.
\end{displaymath}
\begin{definition}
 \label{df1-40}
 Let $m$ be a positive integer and let $\balpha$ be a multi-index of weight
 $m$. Then we define
 \begin{displaymath}
  \balpha^{*} = \calS_m^{-1}(\calS_m(\balpha)^c),
 \end{displaymath}
 where $\calS_m(\balpha)^c$ denotes the complement of $\calS_m(\balpha)$ in 
 the set $\{1,2,\ldots,m-1\}$.
\end{definition}
For example, we have
\begin{displaymath}
 (2,2)^{*} = (1,2,1),\quad (1,1,2)^{*} = (3,1) \quad \text{and} \quad
 (4)^{*} = (1,1,1,1)
\end{displaymath}
from the diagrams
\begin{displaymath}
 {\setlength{\arraycolsep}{0pt}
 \begin{array}{ccccccc}
          &        &        &\downarrow&        &        &         \\
  \bigcirc&        &\bigcirc&          &\bigcirc&        &\bigcirc \\
          &\uparrow&        &          &        &\uparrow&
 \end{array}}\, , \qquad
 {\setlength{\arraycolsep}{0pt}
 \begin{array}{ccccccc}
          &\downarrow&        &\downarrow&        &        &         \\
  \bigcirc&          &\bigcirc&          &\bigcirc&        &\bigcirc \\
          &          &        &          &        &\uparrow&  
 \end{array}} \qquad \text{and} \qquad
 {\setlength{\arraycolsep}{0pt}
 \begin{array}{ccccccc}
          &        &        &        &        &        &         \\
  \bigcirc&        &\bigcirc&        &\bigcirc&        &\bigcirc \\
          &\uparrow&        &\uparrow&        &\uparrow&
 \end{array}}\,,
\end{displaymath}
where the lower arrows are in the complementary slots to the upper arrows.
Let $\balpha=(\alphavec)$ be a multi-index.
We define the multi-index $\balpha^{\tau}$ to be 
$(\alpha_s,\ldots,\alpha_1)$ and put
\begin{equation}
 \balpha^{\backprime} := (\balpha^{*})^{\tau} = (\balpha^{\tau})^{*}. 
  \label{eq1-10}
\end{equation}
If $\balpha \neq (1)$, 
we define the multi-indices ${}^{-}\!\balpha$ and $\balpha^{-}$ by
\begin{displaymath}
 {}^{-}\!\balpha =
 \begin{cases}
  (\alpha_1-1,\alpha_2,\ldots,\alpha_s) & \text{if $\alpha_1 \ge 2$} \\
  (\alpha_2,\ldots,\alpha_s) & \text{if $\alpha_1 = 1$} 
 \end{cases}
\end{displaymath}
and
\begin{displaymath}
 \balpha^{-} = 
 \begin{cases}
  (\alpha_1,\ldots,\alpha_{s-1},\alpha_s-1) & \text{if $\alpha_s \ge 2$} \\
  (\alpha_1,\ldots,\alpha_{s-1}) & \text{if $\alpha_s = 1$},
 \end{cases}
\end{displaymath}
respectively. Then we have
\begin{gather*}
 (\balpha^{-})^{*} = (\balpha^{*})^{-}, \quad ({}^{-}\!\balpha)^{*} =
 {}^{-}\!(\balpha^{*}),\\
 (\balpha^{-})^{\tau} = {}^{-}\!(\balpha^{\tau}), \quad ({}^{-}\!\balpha)^{\tau} =
 (\balpha^{\tau})^{-},
\end{gather*}
and therefore,
\begin{equation}
 (\balpha^{-})^{\backprime} = {}^{-}\!(\balpha^{\backprime}), \quad 
  ({}^{-}\!\balpha)^{\backprime}
 = (\balpha^{\backprime})^{-}. \label{eq1-20}
\end{equation}
\par
Next, we turn to finite multiple harmonic sums.
Let $\N$ denotes the set of all non-negative integers.
We first define two operators on the space $\C^{\N}$ of complex-valued 
sequences which are used in the sequel.
\begin{definition}
 \label{df1-10}
 We define the difference operator $\Delta \colon \C^{\N} \to \C^{\N}$
 by putting
 \begin{displaymath}
  (\Delta a)(n) = a(n) - a(n+1)
 \end{displaymath}
 for any $a \in \C^{\N}$ and any $n \in \N$.
\end{definition}
We denote the composition of $\Delta$ with itself $n$ times by
$\Delta^n$.
If $n=0$, we define $\Delta^0$ to be the identity map on $\C^{\N}$.
\begin{definition}
 \label{df1-20}
 We define the inversion operator $\nabla \colon \C^{\N} \to \C^{\N}$ by
 putting
 \begin{displaymath}
  (\nabla a)(n) = (\Delta^n a)(0)
 \end{displaymath}
 for any $a \in \C^{\N}$ and any $n \in \N$.
\end{definition}
Finite multiple harmonic sums are partial sums of multiple zeta
value series.
In this paper, we consider the following finite nested sum
\begin{displaymath}
 S_{\balpha}(n) = \sum_{n>n_1 \ge \cdots \ge n_s \ge 0}
 \frac{1}{(n_1+1)^{\alpha_1} \cdots (n_s+1)^{\alpha_s}}
\end{displaymath}
for any multi-index $\balpha$ and any non-negative integer $n$.
We note that $S_{\balpha}(0)=0$.
If we calculate the difference of the sequence $S_{\balpha}$,
we obtain
\begin{equation}
 (\Delta S_{\balpha})(n) = 
 \frac{-1}{(n+1)^{\alpha_1}} S_{\alpha_2,\ldots,\alpha_s}(n+1), \label{eq1-30}
\end{equation}
where we put $S_{\alpha_2,\ldots,\alpha_s}(n+1)=1$ in case $s=1$.
\section{Multiple series which interpolate finite multiple harmonic
 sums} \label{sec2}
In this section, we define the multiple series $G_{\balpha}(z)$
for any multi-index $\balpha$ and any complex number $z$ not equal to
negative integers.
Our purpose is to prove that this multiple series interpolates 
the finite multiple harmonic sums 
$\{S_{\balpha^{\backprime}}(n)\}_{n=0}^{\infty}$.
That is, we prove that $G_{\balpha}(n)=S_{\balpha^{\backprime}}(n)$ for any 
non-negative integer $n$.
\par
For any integer $r \ge 1$, positive integers $n_1,\ldots,n_{r}$ 
and $z \in \C \setminus \{-1,-2,\ldots\}$, we write
\begin{align*}
 P_{r}(n_1,n_2,\ldots,n_{r};z) &= \frac{1}{(n_1+z) n_2 \cdots n_{r}}
 \intertext{and}
 \tilde{P}_{r}(n_1,n_2,\ldots,n_{r};z) &= 
 \left(\frac{1}{n_1} - \frac{1}{n_1+z} \right) \frac{1}{n_2 \cdots n_{r}}.
\end{align*}
We define $G_{\balpha}(z)$, for any multi-index $\balpha=(\alphavec)$ and 
$z \in \C \setminus \{-1,-2,\ldots\}$, to be
\begin{displaymath}
 \sum_{\substack{\:\:\:m_1 > \cdots > n_1\\ \ge m_2 > \cdots > n_2\\ \cdots\\
 \ge m_s > \cdots > n_s}}
 \tilde{P}_{\alpha_1}(m_1,\ldots,n_1;z)
 P_{\alpha_2}(m_2,\ldots,n_2;z) \cdots
 P_{\alpha_s}(m_s,\ldots,n_s;z),
\end{displaymath}
where the sum is taken over all positive integers which satisfy the condition.
Examples are given in Section \ref{sec0}.
\par
It is easily seen that the multiple series $G_{\balpha}(z)$ is a
meromorphic function in the whole complex plane 
with poles at most at $z=-\alpha_s,-\alpha_s-1,-\alpha_s-2,\ldots.$
We note that $G_{\balpha}(0) = 0$.
In this section, we consider only the case $z \in \N$ and regard
$G_{\balpha}$ as a sequence.
The following Proposition \ref{prop2-10} is the key to prove 
Proposition \ref{th2-20}, which is the purpose of this section,
and the proof will be given later.
\begin{proposition}
 \label{prop2-10}
 Let $\balpha=(\alphavec)$ be a multi-index with 
 $|\balpha| \ge 2$ and let $n \in \N$. \\
 $(\mathrm{i})$ If $\alpha_s \ge 2$ then we have
 \begin{displaymath}
  (\Delta G_{\balpha})(n) = 
  \frac{-1}{n+1} G_{\balpha^{-}}(n+1). %\label{eq2-10}
 \end{displaymath}
 $(\mathrm{ii})$ If $\alpha_s = 1$ then we have
 \begin{displaymath}
  (\Delta G_{\balpha})(n) = 
  \frac{1}{n+1} (\Delta G_{\balpha^{-}})(n).
 \end{displaymath}
\end{proposition}
Proposition \ref{th2-20} is easily proved by induction 
using Proposition \ref{prop2-10}.
\begin{proposition}
 \label{th2-20}
 Let $\balpha$ be a multi-index. Then we have
 \begin{displaymath}
  G_{\balpha}(n) = S_{\balpha^{\backprime}}(n)
 \end{displaymath}
 for any $n \in \N$.
\end{proposition}
\begin{proof}
 The proof is by induction on $|\balpha|$. The case $|\balpha| = 1$ 
 is easily seen.
 We assume that $|\balpha| \ge 2$. Let $\balpha=(\alphavec)$ and
 $\balpha^{\backprime} = (\alpha^{\backprime}_1,\ldots,\alpha^{\backprime}_t)$.
 If $\alpha_s \ge 2$ then we have $\alpha^{\backprime}_1 = 1$. 
 By Proposition \ref{prop2-10}, the hypothesis of induction, 
 (\ref{eq1-20}) and (\ref{eq1-30}), we obtain
 \begin{displaymath}
  (\Delta G_{\balpha})(n) = \frac{-1}{n+1} G_{\balpha^{-}}(n+1)
  = \frac{-1}{n+1} S_{{}^{-}\!(\balpha^{\backprime})}(n+1)
  = (\Delta S_{\balpha^{\backprime}})(n).
 \end{displaymath}
 Since $G_{\balpha}(0) = S_{\balpha^{\backprime}}(0) = 0$, it holds that
 $G_{\balpha}(n) = S_{\balpha^{\backprime}}(n)$ for any $n \in \N$. If
 $\alpha_s = 1$ then we have $\alpha^{\backprime}_1 \ge 2$. 
 Also in this case, we obtain 
 $G_{\balpha}(n) = S_{\balpha^{\backprime}}(n)$ 
 for any $n \in \N$ by a similar argument.
 We have thus completed the proof.
\end{proof}
In the rest of this section, we prove Proposition \ref{prop2-10}.
We first give another description of the multiple series $G_{\balpha}(z)$.
For any integer $r \ge 1$,
positive integers $n_1,\ldots,n_r$ and 
$z \in \C \setminus \{-1,-2,\ldots\}$, we write
\begin{align*}
 Q_{r}(n_1,n_2,\ldots,n_{r};z) &= 
 \frac{1}{(n_1+z)(n_2+z) \cdots (n_{r}+z)},\\
 \tilde{Q}_{r}(n_1,n_2,\ldots,n_{r};z) &= 
 \left(\frac{1}{n_1} - \frac{1}{n_1+z} \right) 
 \frac{1}{(n_2+z) \cdots (n_{r}+z)}
 \intertext{and}
 R_{r}(n_1,n_2,\ldots,n_{r}) &= 
 \frac{1}{n_1 n_2 \cdots n_{r}}.
\end{align*}
In addition, we define
\begin{displaymath}
 Q_{0} = R_{0} = 1.
\end{displaymath}
For any multi-index $\balpha=(\alphavec)$,
integers $\beta_1,\ldots,\beta_t \ge 1$,
$\beta'_1,\ldots,\beta'_{t-1} \ge 1$, $\beta'_t \ge 0$ 
are uniquely determined by
\begin{displaymath}
 \balpha = (\underbrace{1,\ldots,1,\beta'_1+1}_{\beta_1},
 \underbrace{1,\ldots,1,\beta'_2+1}_{\beta_2}, \ldots ,
 \underbrace{1,\ldots,1,\beta'_t+1}_{\beta_t}).
\end{displaymath}
By definition, the multiple series $G_{\balpha}(z)$ is written as
\begin{multline*}
 \sum_{\substack{\:\:\:k_1 \ge \cdots \ge l_1 > k'_1 > \cdots > l'_1\\ \ge k_2
 \ge \cdots \ge l_2 > k'_2 > \cdots > l'_2\\ 
 \cdots\cdots \\ 
 \ge k_t \ge \cdots \ge l_t
 > k'_t > \cdots > l'_t}}
 \tilde{Q}_{\beta_1}(k_1,\ldots,l_1;z) 
 R_{\beta'_1}(k'_1,\ldots,l'_1) \\
 \times Q_{\beta_2}(k_2,\ldots,l_2;z) 
 R_{\beta'_2}(k'_2,\ldots,l'_2)
 \cdots Q_{\beta_t}(k_t,\ldots,l_t;z) 
 R_{\beta'_t}(k'_t,\ldots,l'_t),
\end{multline*}
where the sum is taken over all positive integers which satisfy the
condition.
\par
We define more general nested sums, which appear 
in the process of the proof of Proposition \ref{prop2-10}.
Let $\mathcal{C}$ denotes the set of matrices
\begin{displaymath}
 \begin{pmatrix}
            & \gamma_1 & \square'_1 & \gamma'_1 \\
  \square_2 & \gamma_2 & \square'_2 & \gamma'_2 \\
            & \cdots \\
  \square_p & \gamma_p & \square'_p & \gamma'_p \\
 \end{pmatrix}
\end{displaymath}
which satisfy the following conditions:
\begin{itemize}
 \item The number of columns is $4$ and the number of rows is greater
       than or equal to $1$.
 \item The $(1,1)$ entry is empty.
       The other entries in the first and the third columns are the symbols
       $>$, $\ge$ or $=$.
 \item The entries in the second and the fourth columns are non-negative 
       integers and $\gamma_1 \ge 1$.
 \item For $2 \le i \le p$ satisfying $\gamma_i = 0$, we have
       $\square_i = \square'_i$.
 \item For $1 \le i \le p-1$ satisfying $\gamma'_i = 0$, we have
       $\square'_i = \square_{i+1}$.
\end{itemize}
We denote by $\Q\mathcal{C}$ the $\Q$-vector space with basis
$\mathcal{C}$. 
Let $n \in \N$.
We define a $\Q$-linear mapping $\Phi_n \colon \Q\mathcal{C} \to \R$ by putting
\begin{multline*}
 \Phi_n(C) =
 \sum_{\substack{\phantom{\square_1} k_1 \ge \cdots \ge l_1 \square'_1 
 k'_1 > \cdots > l'_1\\ 
 \square_2 k_2 \ge \cdots \ge l_2 \square'_2
 k'_2 > \cdots > l'_2\\ 
 \cdots\cdots \\ 
 \square_p k_p \ge \cdots \ge l_p
 \square'_p k'_p > \cdots > l'_p}}
 \tilde{Q}_{\gamma_1}(k_1,\ldots,l_1;n) R_{\gamma'_1}(k'_1,\ldots,l'_1) \\
 \times Q_{\gamma_2}(k_2,\ldots,l_2;n) R_{\gamma'_2}(k'_2,\ldots,l'_2)
 \cdots Q_{\gamma_p}(k_p,\ldots,l_p;n) R_{\gamma'_p}(k'_p,\ldots,l'_p)
\end{multline*}
for any
\begin{displaymath}
 C =
 \begin{pmatrix}
            & \gamma_1 & \square'_1 & \gamma'_1 \\
  \square_2 & \gamma_2 & \square'_2 & \gamma'_2 \\
  & \cdots \\
  \square_p & \gamma_p & \square'_p & \gamma'_p \\
 \end{pmatrix}
 \in \mathcal{C},
\end{displaymath}
where the sum is taken over all integers which satisfy the condition.
If $\gamma_p \ne 0$ and $\gamma'_p = 0$, for example,
the sum $\Phi_n(C)$ is defined independently of $\square'_p$. 
Now we prove Proposition \ref{prop2-10}.
\par
\medskip
\noindent
\emph{Proof of Proposition \ref{prop2-10}}
 Let $\balpha=(\alphavec)$ be a multi-index with $|\balpha| \ge 2$ 
 and determine integers 
 $\beta_1,\ldots,\beta_t,\beta'_1,\ldots,\beta'_{t-1} \ge 1$, 
 $\beta'_t \ge 0$ by
\begin{displaymath}
 \balpha = (\underbrace{1,\ldots,1,\beta'_1+1}_{\beta_1},
 \underbrace{1,\ldots,1,\beta'_2+1}_{\beta_2}, \ldots ,
 \underbrace{1,\ldots,1,\beta'_t+1}_{\beta_t}).
\end{displaymath}
We prove only the case $\alpha_1=1$ and $\alpha_s \ge 2$
(i.e. $\beta_1 \ge 2$ and $\beta'_t \ge 1$).
The other cases 
(i.e. $\alpha_1 \ge 2$ and $\alpha_s \ge 2$,
$\alpha_1 = 1$ and $\alpha_s = 1$, $\alpha_1 \ge 2$ and $\alpha_s = 1$)
can be proved similarly.
We first give a lemma which is needed later. 
Its proof is immediate from calculations similar to the following:
Let $n_1,\ldots,n_r,m_1,\ldots,m_q$ be positive integers and $n \in \N$.
If $n_{r} = m_1$ then we have
 \begin{multline*}
  R_{r}(n_1,\ldots,n_{r}) Q_{q}(m_1,\ldots,m_q;n+1) \\
  \shoveleft{= \frac{1}{n+1} \times \frac{1}{n_{1} \cdots n_{r-1}}  
  \left(\frac{1}{n_{r}} - \frac{1}{m_1+n+1}\right)
  \frac{1}{(m_2+n+1) \cdots (m_q+n+1)} } \\
  \shoveleft{= \frac{1}{n+1} \Bigl\{R_r(n_1,\ldots,n_r)
  Q_{q - 1}(m_2,\ldots,m_q;n+1) \Bigr. } \\
  \Bigl. - R_{r - 1}(n_1,\ldots,n_{r-1}) Q_q(m_1,\ldots,m_q;n+1) \Bigr\}.
 \end{multline*}
\begin{lemma}
 \label{lem2-40}
 Let 
 \begin{displaymath}
  C =
  \begin{pmatrix}
             & \gamma_1 & \square'_1 & \gamma'_1 \\
   \square_2 & \gamma_2 & \square'_2 & \gamma'_2 \\
   & \cdots \\
   \square_p & \gamma_p & \square'_p & \gamma'_p
  \end{pmatrix}
  \in \mathcal{C}
 \end{displaymath}
 with $\gamma_1 \ge 2$, 
 $\gamma_2,\ldots,\gamma_p,\gamma'_1,\ldots,\gamma'_p \ge 1$
 and let $n \in \N$. \\
 $(\mathrm{i})$ Let $2 \le i \le p$. 
 If $\square_i$ equals to $=$ then we have
 \begin{displaymath}
  C \equiv \frac{1}{n+1} 
  \begin{pmatrix}
             & \gamma_1 & \square'_1 & \gamma'_1 \\
   \square_2 & \gamma_2 & \square'_2 & \gamma'_2 \\
             & \cdots \\
   \triangle_i & \gamma_i - 1 & \square'_i & \gamma'_i \\
             & \cdots \\
   \square_p & \gamma_p & \square'_p & \gamma'_p
  \end{pmatrix}
  - \frac{1}{n+1}
  \begin{pmatrix}
             & \gamma_1 & \square'_1 & \gamma'_1 \\
   \square_2 & \gamma_2 & \square'_2 & \gamma'_2 \\
             & \cdots \\
   \square_{i-1} & \gamma_{i-1} & \square'_{i-1} & \gamma'_{i-1} -1 \\
   \bigcirc_i  & \gamma_i & \square'_i & \gamma'_i \\
             & \cdots \\
   \square_p & \gamma_p & \square'_p & \gamma'_p 
  \end{pmatrix}  
 \end{displaymath}
 modulo the kernel of the linear map 
 $\Phi_{n+1} \colon \Q \mathcal{C} \to \R$,
 where $\triangle_i$ equals to $\ge$ if $\gamma_i \ge 2$ and to $\square'_i$
 if $\gamma_i = 1;$ $\bigcirc_i$ equals to $>$ if $\gamma'_{i-1} \ge 2$ and
 to $\square'_{i-1}$ if $\gamma'_{i-1}=1$.\\
 $(\mathrm{ii})$ Let $1 \le i \le p$. 
 If $\square'_i$ equals to $=$ then we have
 \begin{displaymath}
  C \equiv \frac{1}{n+1} 
  \begin{pmatrix}
             & \gamma_1 & \square'_1 & \gamma'_1 \\
   \square_2 & \gamma_2 & \square'_2 & \gamma'_2 \\
             & \cdots \\
   \square_i & \gamma_i - 1 & \triangle'_i & \gamma'_i \\
             & \cdots \\
   \square_p & \gamma_p & \square'_p & \gamma'_p \\
  \end{pmatrix}
  - \frac{1}{n+1}
  \begin{pmatrix}
             & \gamma_1 & \square'_1 & \gamma'_1 \\
   \square_2 & \gamma_2 & \square'_2 & \gamma'_2 \\
             & \cdots \\
   \square_i & \gamma_i & \bigcirc'_i   & \gamma'_i - 1 \\
             & \cdots \\
   \square_p & \gamma_p & \square'_p & \gamma'_p 
  \end{pmatrix}  
 \end{displaymath}
 modulo the kernel of the linear map 
 $\Phi_{n+1} \colon \Q \mathcal{C} \to \R$,
 where $\triangle'_i$ equals to $\ge$ if $\gamma_i \ge 2$ and to $\square_i$
 if $\gamma_i = 1;$ $\bigcirc'_i$ equals to $>$ if $\gamma'_{i} \ge 2$ and
 to $\square_{i+1}$ if $\gamma'_{i}=1$. 
\end{lemma}
In the following, we calculate $(\Delta G_{\balpha})(n)$.
By definition, we have
\begin{multline}
 \Phi_n(
 \begin{pmatrix}
       & \beta_1 & > & \beta'_1 \\
   \ge & \beta_2 & > & \beta'_2 \\
       & \cdots \\
   \ge & \beta_t & > & \beta'_t
 \end{pmatrix}
 ) \\
 \shoveleft{=\sum_{\substack{\phantom{>} k_1 \ge \cdots \ge l_1 
 \ge k'_1 > \cdots > l'_1\\ 
 > k_2 \ge \cdots \ge l_2 \ge k'_2 > \cdots > l'_2\\ 
 \cdots\cdots\\ 
 > k_t \ge \cdots \ge l_t \ge k'_t > \cdots > l'_t}}
 \tilde{Q}_{\beta_1}(k_1+1,\ldots,l_1+1;n) 
 R_{\beta'_1}(k'_1,\ldots,l'_1) } \\
 \times Q_{\beta_2}(k_2+1,\ldots,l_2+1;n) 
 R_{\beta'_2}(k'_2,\ldots,l'_2)
 \cdots Q_{\beta_t}(k_t+1,\ldots,l_t+1;n) 
 R_{\beta'_t}(k'_t,\ldots,l'_t). \label{eq2-20}
\end{multline}
Since each term
\begin{multline*}
 \tilde{Q}_{\beta_1}(k_1+1,m_1+1,\ldots,l_1+1;n) 
 R_{\beta'_1}(k'_1,\ldots,l'_1) \\
 \shoveright{\times Q_{\beta_2}(k_2+1,\ldots,l_2+1;n) 
 R_{\beta'_2}(k'_2,\ldots,l'_2) 
 \cdots 
 Q_{\beta_t}(k_t+1,\ldots,l_t+1;n) 
 R_{\beta'_t}(k'_t,\ldots,l'_t) } \\
\end{multline*}
is written as
\begin{multline*}
 \left(\frac{1}{k_1+1} - \frac{1}{k_1} + 
 \frac{1}{k_1} - \frac{1}{k_1+n+1} \right)
 Q_{\beta_1-1}(m_1,\ldots,l_1;n+1) 
 R_{\beta'_1}(k'_1,\ldots,l'_1) \\
 \times Q_{\beta_2}(k_2,\ldots,l_2;n+1) 
 R_{\beta'_2}(k'_2,\ldots,l'_2) 
 \cdots 
 Q_{\beta_t}(k_t,\ldots,l_t;n+1) 
 R_{\beta'_t}(k'_t,\ldots,l'_t),
\end{multline*}
the right-hand side of (\ref{eq2-20}) is equal to
\begin{displaymath}
 \Phi_{n+1}(
 \begin{pmatrix}
       & \beta_1 & \ge & \beta'_1 \\
   >   & \beta_2 & \ge & \beta'_2 \\
       & \cdots \\
   >   & \beta_t & \ge & \beta'_t
 \end{pmatrix}
 ) - \frac{1}{n+1} \Phi_{n+1}(
 \begin{pmatrix}
       & \beta_1 - 1 & \ge & \beta'_1 \\
   >   & \beta_2     & \ge & \beta'_2 \\
       & \cdots \\
   >   & \beta_t     & \ge & \beta'_t
 \end{pmatrix}
 ).
\end{displaymath}
Therefore we see that
\begin{multline*}
 (\Delta G_{\balpha})(n) = \Phi_{n+1}(
 \begin{pmatrix}
       & \beta_1 & \ge & \beta'_1 \\
   >   & \beta_2 & \ge & \beta'_2 \\
       & \cdots \\
   >   & \beta_t & \ge & \beta'_t
 \end{pmatrix}
 - 
 \begin{pmatrix}
       & \beta_1 & > & \beta'_1 \\
   \ge & \beta_2 & > & \beta'_2 \\
       & \cdots \\
   \ge & \beta_t & > & \beta'_t
 \end{pmatrix}
 ) \\
 - \frac{1}{n+1} \Phi_{n+1}(
 \begin{pmatrix}
       & \beta_1 - 1 & \ge & \beta'_1 \\
   >   & \beta_2     & \ge & \beta'_2 \\
       & \cdots \\
   >   & \beta_t     & \ge & \beta'_t
 \end{pmatrix}
 ). %\label{eq2-30}
 \end{multline*}
This equation and the following Lemma \ref{lem2-50} imply 
Proposition \ref{prop2-10}.
\begin{lemma}
 \label{lem2-50}
 Let $p$ be a positive integer and $n \in \N$.
 For any integer's $\gamma_1 \ge 2$,
 $\gamma_2,\ldots,\gamma_p,\gamma'_1,\ldots,\gamma'_p \ge 1$, we have
 \begin{multline*} 
  \begin{pmatrix}
       & \gamma_1 & \ge & \gamma'_1 \\
   >   & \gamma_2 & \ge & \gamma'_2 \\
       & \cdots \\
   >   & \gamma_p & \ge & \gamma'_p
  \end{pmatrix}
   -
  \begin{pmatrix}
       & \gamma_1 & > & \gamma'_1 \\
   \ge & \gamma_2 & > & \gamma'_2 \\
       & \cdots \\
   \ge & \gamma_p & > & \gamma'_p
  \end{pmatrix} \\
  \equiv \frac{1}{n+1} 
  \begin{pmatrix}
       & \gamma_1 - 1 & \ge & \gamma'_1 \\
   >   & \gamma_2 & \ge & \gamma'_2 \\
       & \cdots \\
   >   & \gamma_p & \ge & \gamma'_p
  \end{pmatrix} 
  - \frac{1}{n+1}
  \begin{pmatrix}
       & \gamma_1 & > & \gamma'_1 \\
   \ge & \gamma_2 & > & \gamma'_2 \\
       & \cdots \\
   \ge & \gamma_{p-1} & > & \gamma'_{p-1} \\
   \ge & \gamma_p & > & \gamma'_p - 1
  \end{pmatrix} 
 \end{multline*}
 modulo the kernel of the linear map
 $\Phi_{n+1} \colon \Q \mathcal{C} \to \R$. 
\end{lemma}
\begin{proof}
 In the proof, for example, we denote
 \begin{displaymath}
  \begin{pmatrix}
       & \gamma_1     & \ge & \gamma'_1 \\
   \ge & \gamma_2     & \ge & \gamma'_2 \\
       & \cdots \\
   \ge & \gamma_{i-1} & \ge & \gamma'_{i-1} \\
   =   & \gamma_i     & \ge & \gamma'_i \\
   >   & \gamma_{i+1} & \ge & \gamma'_{i+1} \\
       & \cdots \\
   >   & \gamma_p     & \ge & \gamma'_p
  \end{pmatrix} \quad \text{by} \quad
  \begin{pmatrix}
   \ge &          & \ge \\
   =   & \gamma_i & \ge & \gamma'_i \\
   >   &          & \ge
  \end{pmatrix}
 \end{displaymath}
 for abbreviation. 
 We easily see that
 \begin{multline*}
  \begin{pmatrix}
       & \gamma_1 & \ge & \gamma'_1 \\
   \ge & \gamma_2 & \ge & \gamma'_2 \\
       & \cdots \\
   \ge & \gamma_p & \ge & \gamma'_p
  \end{pmatrix}
  \equiv
  \begin{pmatrix}
       & \gamma_1 & \ge & \gamma'_1 \\
   >   & \gamma_2 & \ge & \gamma'_2 \\
       & \cdots \\
   >   & \gamma_p & \ge & \gamma'_p
  \end{pmatrix}  
  + \sum_{i=2}^p
  \begin{pmatrix}
   \ge &          & \ge \\
   =   & \gamma_i & \ge & \gamma'_i \\
   >   &          & \ge
  \end{pmatrix} \\
  \equiv
  \begin{pmatrix}
       & \gamma_1 & \ge & \gamma'_1 \\
   >   & \gamma_2 & \ge & \gamma'_2 \\
       & \cdots \\
   >   & \gamma_p & \ge & \gamma'_p
  \end{pmatrix}  
  + \sum_{i=2}^p
  \begin{pmatrix}
   \ge &          & > \\
   =   & \gamma_i & \ge & \gamma'_i \\
   >   &          & \ge
  \end{pmatrix}  
  + \sum_{p \ge i > j \ge 1}
  \begin{pmatrix}
   \ge &          & > \\
   \ge & \gamma_j & = & \gamma'_j \\
   \ge &          & \ge \\
   =   & \gamma_i & \ge & \gamma'_i \\
   >   &          & \ge
  \end{pmatrix}    
 \end{multline*}
 and
 \begin{multline*}
  \begin{pmatrix}
       & \gamma_1 & \ge & \gamma'_1 \\
   \ge & \gamma_2 & \ge & \gamma'_2 \\
       & \cdots \\
   \ge & \gamma_p & \ge & \gamma'_p
  \end{pmatrix}
  \equiv
  \begin{pmatrix}
       & \gamma_1 & > & \gamma'_1 \\
   \ge & \gamma_2 & > & \gamma'_2 \\
       & \cdots \\
   \ge & \gamma_p & > & \gamma'_p
  \end{pmatrix}  
  + \sum_{i=1}^p
  \begin{pmatrix}
   \ge &          & > \\
   \ge & \gamma_i & = & \gamma'_i \\
   \ge &          & \ge
  \end{pmatrix} \\
  \equiv
  \begin{pmatrix}
       & \gamma_1 & > & \gamma'_1 \\
   \ge & \gamma_2 & > & \gamma'_2 \\
       & \cdots \\
   \ge & \gamma_p & > & \gamma'_p
  \end{pmatrix}  
  + \sum_{i=1}^p
  \begin{pmatrix}
   \ge &          & > \\
   \ge & \gamma_i & = & \gamma'_i \\
   >   &          & \ge
  \end{pmatrix}  
  + \sum_{p \ge j > i \ge 1}
  \begin{pmatrix}
   \ge &          & > \\
   \ge & \gamma_i & = & \gamma'_i \\
   \ge &          & \ge \\
   =   & \gamma_j & \ge & \gamma'_j \\
   >   &          & \ge
  \end{pmatrix}.
 \end{multline*}
 From these we obtain the congruence
 \begin{multline*}
  \begin{pmatrix}
       & \gamma_1 & \ge & \gamma'_1 \\
   >   & \gamma_2 & \ge & \gamma'_2 \\
       & \cdots \\
   >   & \gamma_p & \ge & \gamma'_p
  \end{pmatrix}
   -
  \begin{pmatrix}
       & \gamma_1 & > & \gamma'_1 \\
   \ge & \gamma_2 & > & \gamma'_2 \\
       & \cdots \\
   \ge & \gamma_p & > & \gamma'_p
  \end{pmatrix} \\
  \equiv \sum_{i=1}^p
  \begin{pmatrix}
   \ge & & > \\
   \ge & \gamma_i & = & \gamma'_i \\
   > & & \ge
  \end{pmatrix}  
  - \sum_{i=2}^p
  \begin{pmatrix}
   \ge & & > \\
   = & \gamma_i & \ge & \gamma'_i \\
   > & & \ge
  \end{pmatrix},
 \end{multline*}
 which together with Lemma \ref{lem2-40} proves the lemma.
\end{proof}
\section{Multiple series expressions for the Newton series which
 interpolate finite multiple harmonic sums} \label{sec3}
We start with the definition of the Newton series.
Let $a \in \C^{\N}$ be a sequence and $z$ a complex variable.
The series
\begin{displaymath}
 f(z) = \sum_{n=0}^{\infty}(-1)^n (\nabla a)(n) \binom{z}{n}, \quad
 \binom{z}{n} = \frac{z(z-1)\cdots(z-n+1)}{n!}
\end{displaymath}
is called the Newton series which interpolate a sequence $a$,
where the definition of $\nabla a$ is given in Definition \ref{df1-20}.
In fact, we have $f(n) = a(n)$ for any $n \in \N$. 
(See \cite[Section 4]{K}.)
For the convergence of the Newton series, the following is known
(see for instance \cite[Section 10.6]{M}):
there exists some $\rho \in \R \cup \{\pm \infty\}$ such that 
the series $f(z)$ converges for any $z \in \C$ with $\re z > \rho$
and diverges for any $z \in \C \setminus \N$ with $\re z < \rho$.
Clearly, such a $\rho$ is uniquely determined.
This $\rho$ is called the abscissa of convergence of the Newton series
$f(z)$.
In the half-plane $\re z > \rho$, the function $f(z)$ is holomorphic.
\par
Let $\balpha$ be a multi-index.
In this section, we consider the Newton series
\begin{displaymath}
 F_{\balpha}(z) = \sum_{n=0}^{\infty}(\nabla S_{\balpha})(n) \binom{z}{n}
\end{displaymath}
which interpolate the sequence $S_{\balpha}$.
The following Proposition \ref{prop3-10} provides 
the abscissa of convergence of $F_{\balpha}(z)$.
\begin{proposition}
 \label{prop3-10}
 Let $\balpha$ be a multi-index and let 
 $\balpha^{*}=(\alpha^{*}_1,\ldots,\alpha^{*}_t)$.
 Then the abscissa of convergence of $F_{\balpha}(z)$ is equal to 
 $-\alpha^{*}_1$.
\end{proposition}
For a proof of Proposition \ref{prop3-10}, see \cite[Proposition 5.1]{K}.
(Although Proposition 5.1 in \cite{K} is false for $r=0$, its proof is valid.)
Our objective in this paper is to prove that
\begin{displaymath}
 F_{\balpha}(z) = G_{\balpha^{\backprime}}(z)
\end{displaymath}
for any $\re z > -\alpha_1^{*}$.
For this purpose we use the following proposition.
\begin{proposition}
 \label{prop3-20}
 Let $a$, $b \in \C^{\N}$ be sequences and let $\rho \in \R$. 
 We suppose that the Newton series
 \begin{displaymath}
  f(z) = \sum_{n=0}^{\infty} (-1)^n (\nabla a) \binom{z}{n}
  \quad \text{and} \quad
  g(z) = \sum_{n=0}^{\infty} (-1)^n (\nabla b) \binom{z}{n}
 \end{displaymath}
 converge for any $z \in \C$ with $\re z > \rho$.
 If there exists $N \in \N$ such that $f(n)=g(n)$ for any $n \in \N$ with
 $n \ge N$, then we have $f(z)=g(z)$ for any $\re z > \rho$.
\end{proposition}
\begin{proof}
 Since the sequence $c(n) := a(n)-b(n)$ vanishes for any $n \ge N$,
 there exists some polynomial $P$ with degree at most $N$ such that
 $(\nabla c)(n) = P(n)$ for any $n \in \N$.
 Therefore the assertion follows from the following fact:
 For a polynomial $Q$ with degree $d$, the Newton series
 \begin{displaymath}
  \sum_{n=0}^{\infty} (-1)^n Q(n) \binom{z}{n}
 \end{displaymath}
 has the abscissa of convergence $d$ and vanishes in the region 
 $\re z > d$ (see \cite[Section 10.6]{M}).
\end{proof}
If the function $G_{\balpha^{\backprime}}(z)$ is expressed
by a Newton series in some half-plane $\re z > \sigma$ 
($\sigma \in \R$),
we obtain the equation
\begin{displaymath}
 F_{\balpha}(z) = G_{\balpha^{\backprime}}(z)
\end{displaymath}
for any $\re z > \max\{-\alpha_1^{*},\sigma\}$
by Propositions \ref{th2-20} and \ref{prop3-20}.
This equation is still valid for any $\re z > -\alpha_1^{*}$ 
since the functions $F_{\balpha}(z)$ and $G_{\balpha^{\backprime}}(z)$
are both holomorphic in $\re z > -\alpha_1^{*}$.
Therefore we have only to show that the function 
$G_{\balpha^{\backprime}}(z)$ is expressed by a Newton series in some
half-plane.
This follows easily from Proposition \ref{prop3-30}.
Before stating Proposition \ref{prop3-30}, 
we remark some properties of a function used in the proposition.
The following function
\begin{displaymath}
 \psi(\theta) = \cos\theta \log(2\cos\theta) + \theta\sin\theta,
 \quad |\theta| \le \frac{\pi}{2} 
\end{displaymath}
is an even function and monotone increasing on the interval $[0,\pi/2]$.
Moreover we have
\begin{displaymath}
 \psi(0) = \log 2 \quad \text{and} \quad 
 \psi \left(\frac{\pi}{2} \right) = \frac{\pi}{2}.
\end{displaymath}
\begin{proposition}
 \label{prop3-30}
 Let $f(z)$ be a function which is holomorphic in the half-plane
 $\re z \ge \alpha$ $(\alpha \in \R)$.
 Let $A>0$ and $\beta$ be real numbers and 
 let a real function $\varepsilon(r)$ on $[0,\infty)$
 tend to zero as $r \to \infty$.
 If the inequality
 \begin{displaymath}
  \left|f(\alpha + r e^{i\theta}) \right| \le A e^{r \psi(\theta)} 
  (1+r)^{\beta + \varepsilon(r)}
 \end{displaymath}
 holds for any $r \ge 0$ and $|\theta| \le \pi/2$,
 the function $f(z)$ can be expressed by a Newton series
 in the half-plane $\re z > \max\{\alpha, \beta - 1/2\}$.
\end{proposition}
\begin{proof}
 See \cite[Chapter V]{N}.
\end{proof}
Let $\balpha$ be a multi-index.
Since we have
\begin{displaymath}
 \left|\tilde{Q}_r(n_1,n_2,\ldots,n_r;z)\right|
 = \left|\frac{z}{n_1(n_1+z)(n_2+z)\cdots(n_r+z)}\right|
 \le \frac{|z|}{n_1^2 n_2 \cdots n_r}
\end{displaymath}
if $\re z \ge 0$, there exists some constant $A > 0$ such that the
inequality
\begin{displaymath}
 \left|G_{\balpha}(z)\right| \le A |z|
\end{displaymath}
holds for any $\re z \ge 0$.
According to Proposition \ref{prop3-30}, the function $G_{\balpha}(z)$
is expressed by a Newton series in some half-plane.
Consequently, we obtain the following theorem which is our main result.
\begin{theorem}
 \label{th3-40}
 Let $\balpha$ be a multi-index and let 
 $\balpha^{*}=(\alpha_1^{*},\ldots,\alpha_t^{*})$.
 Then we have
 \begin{displaymath}
  F_{\balpha}(z) = G_{\balpha^{\backprime}}(z)
 \end{displaymath}
 for any $\re z > - \alpha_1^{*}$.
\end{theorem}
\section{A formula for multiple zeta values}
In this section, we evaluate the $k$-th derivatives at $z=0$ of the functions
$F_{\balpha}(z)$ and $G_{\balpha}(z)$  
for any multi-index $\balpha$ and any integer $k \ge 1$.
These derivatives are all $\Q$-linear combinations of 
multiple zeta values (MZV's).
Since we have $F_{\balpha}^{(k)}(0) = G_{\balpha^{\backprime}}^{(k)}(0)$
by Theorem \ref{th3-40}, we can obtain $\Q$-linear relations among MZV's.
The relations for $k=1$ are nothing else but the duality for MZV's.
\par
We give some definitions which are used in the sequel.
For any multi-index $\balpha=(\alpha_1,\ldots,\alpha_s)$, we define
a multi-index ${}^{+}\!\balpha$ by
${}^{+}\!\balpha=(\alpha_1+1,\alpha_2,\ldots,\alpha_s)$ 
and put
${}^{+}\!I=\{{}^{+}\!\balpha \,|\, \balpha \in I\}$, 
where $I$ is the set of all multi-indices.
An element of ${}^{+}\!I$ is called an admissible multi-index.
We denote by $V$ and ${}^{+}\!V$ the $\Q$-vector spaces with basis
$I$ and ${}^{+}\!I$, respectively.
\par
We define $\Q$-linear mappings $u$ and $d$ from $V$ to itself by putting
\begin{displaymath}
 u(\balpha) = \sum_{\bbeta \ge \balpha} \bbeta 
\end{displaymath}
and
\begin{displaymath}
 d(\balpha) = \sum_{\bbeta \le \balpha} \bbeta 
\end{displaymath}
for any $\balpha \in I$, respectively.
The partial order $\ge$ on $I$ is defined by setting 
$\bbeta \ge \balpha$ if $\bbeta$ is a refinement of $\balpha$.
(See~\cite[Section 2]{K} for the precise definition.)
For example, we have
\begin{displaymath}
 u(1,3) = (1,3) + (1,2,1) + (1,1,2) + (1,1,1,1)
\end{displaymath}
and
\begin{displaymath}
 d(1,2,3) = (1,2,3) + (3,3) + (1,5) + (6).
\end{displaymath}
We extend the mappings $\ast \colon I \to V$, $\balpha \mapsto \balpha^{*}$
and $\tau \colon I \to V$, $\balpha \mapsto \balpha^{\tau}$
by $\Q$-linearity onto $V$.
It is easily seen that
\begin{equation}
 d \tau = \tau d. \label{eq4-5}
\end{equation}
In addition, we have
\begin{equation}
 \ast d = u \ast. \label{eq4-7}
\end{equation}
(For the proof of this assertion, see~\cite[Proposition~2.3 (i)]{K}.)
\par
We define a $\Q$-bilinear mapping $\circledast \colon V \times V \to {}^{+}\!V$
by putting
\begin{displaymath}
 \balpha \circledast \bbeta = 
 (\alpha_1 + \beta_1) \,\#\, ({}^{-}\!\balpha \ast {}^{-}\!\bbeta)
\end{displaymath}
for any $\balpha=(\alphavec) \in I$ and any $\bbeta=(\betavec) \in I$.
The symbols $\#$ and $*$ denote the concatenation operator and the harmonic
product, respectively.
(See~\cite[Section 2]{K} for the precise definitions.)
For example, we have
\begin{displaymath}
 (\alpha_1,\alpha_2) \circledast (\beta_1) 
 = (\alpha_1 + \beta_1) \,\#\, (\alpha_2) = (\alpha_1 + \beta_1,\alpha_2)
\end{displaymath}
and
\begin{align*}
 (\alpha_1,\alpha_2) \circledast (\beta_1,\beta_2) 
 &= (\alpha_1 + \beta_1) \,\#\, 
 \left\{(\alpha_2,\beta_2)+(\beta_2,\alpha_2)+(\alpha_2+\beta_2)
 \right\} \\
 &= (\alpha_1+\beta_1, \alpha_2, \beta_2) + 
 (\alpha_1+\beta_1, \beta_2, \alpha_2) +
 (\alpha_1+\beta_1, \alpha_2+\beta_2).
\end{align*}
\par
Let $\balpha=(\alphavec)$ be a multi-index.
For any $\bmu \in {}^{+}\!I$ with $l(\bmu)=|\balpha|$,
we define $\zeta_{\balpha}(\bmu) \in \R$ by
\begin{displaymath}
 \zeta_{\balpha}(\bmu) = 
 \sum_{\substack{\phantom{\ge}m_1 > \cdots > n_1\phantom{>0}\\
 \ge m_2 > \cdots > n_2\phantom{>0}\\
 \cdots\\
 \ge m_s > \cdots > n_s > 0}}
 \frac{1}{\underbrace{m_1^{\mu_{11}} \cdots n_1^{\mu_{1\alpha_1}}}_{\alpha_1}
 \underbrace{m_2^{\mu_{21}} \cdots n_2^{\mu_{2\alpha_2}}}_{\alpha_2}
 \cdots
 \underbrace{m_s^{\mu_{s1}} \cdots n_s^{\mu_{s\alpha_s}}}_{\alpha_s}},
\end{displaymath}
where
\begin{displaymath}
 \bmu = (\underbrace{\mu_{11},\ldots,\mu_{1\alpha_1}}_{\alpha_1},
 \underbrace{\mu_{21},\ldots,\mu_{2\alpha_2}}_{\alpha_2},\ldots,
 \underbrace{\mu_{s1},\ldots,\mu_{s\alpha_s}}_{\alpha_s}).
\end{displaymath}
For any $\Q$-linear combination $v$ of the admissible multi-indices of
length $|\balpha|$, we define $\zeta_{\balpha}(v)$ by $\Q$-linearity.
For any admissible multi-index $\bmu=(\muvec)$ we put
\begin{displaymath}
 \zeta(\bmu) = \zeta_{(p)}(\bmu),
\end{displaymath}
which is the usual multiple zeta value, and extend the mapping 
$\zeta \colon {}^{+}\!I \to \R$ by $\Q$-linearity onto ${}^{+}\!V$.
Now, we give the derivatives of the functions 
$F_{\balpha}(z)$ and $G_{\balpha}(z)$ at $z=0$.
\begin{proposition}
 \label{prop4-10}
 For any multi-index $\balpha$ and any integer $k \ge 1$, we have
 \begin{displaymath}
  \frac{F_{\balpha}^{(k)}(0)}{k!} 
  = (-1)^{k-1} \zeta \bigl(d(\balpha^{*}) \circledast 
  (\underbrace{1,\ldots,1}_k) \bigr).
 \end{displaymath}
\end{proposition}
\begin{proof}
 See~\cite[Proposition 5.2]{K}.
\end{proof}
\begin{proposition}
 \label{prop4-20}
 For any multi-index $\balpha=(\alphavec)$ and any integer $k \ge 1$,
 we have
 \begin{displaymath}
  \frac{G_{\balpha}^{(k)}(0)}{k!} = (-1)^{k-1} 
  \sum_{\substack{k_1+\cdots+k_s=k\\ k_1 \ge 1,\, k_2,\,\ldots,\,k_s \ge 0}}
  \zeta_{\balpha}(\underbrace{k_1+1,1,\ldots,1}_{\alpha_1},
  \ldots,\underbrace{k_s+1,1,\ldots,1}_{\alpha_s}).
 \end{displaymath}
\end{proposition}
\begin{proof}
 Let $r \ge 1$ and let $n_1,\ldots,n_r$ be positive integers.
 Then we have
 \begin{equation}
  P_r^{(k)}(n_1,\ldots,n_r;0)
  = \frac{d^k}{dz^k} P_r(n_1,\ldots,n_r;z) \Bigm|_{z=0}
  = \frac{(-1)^k k!}{n_1^{k+1}n_2 \cdots n_r} \label{eq4-10}
 \end{equation}
 for any integer $k \ge 0$.
 In addition, we have
 \begin{equation}
  \tilde{P}_r^{(k)}(n_1,\ldots,n_r;0)
  = \frac{d^k}{dz^k} \tilde{P}_r(n_1,\ldots,n_r;z) \Bigm|_{z=0} 
  = \begin{cases}
     0 & \text{if } k=0 \\
     \dfrac{(-1)^{k-1}k!}{n_1^{k+1}n_2\cdots n_r} & \text{if } k \ge 1.
    \end{cases} \label{eq4-20}
 \end{equation}
 By the Leibniz rule, it holds that
 \begin{multline*}
  G_{\balpha}^{(k)}(0) =
  \sum_{\substack{k_1+\cdots+k_s=k\\ k_1,\,\ldots,\,k_s \ge
  0}}
  \frac{k!}{k_1! \cdots k_s!} \\
  \times
  \sum_{\substack{\phantom{\ge}m_1 > \cdots > n_1\phantom{>0}\\
  \ge m_2 > \cdots > n_2\phantom{>0}\\
  \cdots\\
  \ge m_s > \cdots > n_s > 0}}
  \tilde{P}_{\alpha_1}^{(k_1)}(m_1,\ldots,n_1;0)
  P_{\alpha_2}^{(k_2)}(m_2,\ldots,n_2;0) \cdots
  P_{\alpha_s}^{(k_s)}(m_s,\ldots,n_s;0)
 \end{multline*}
 for any $k \ge 0$, 
 which together with (\ref{eq4-10}) and (\ref{eq4-20})
 gives the desired equality.
\end{proof}
By Propositions \ref{prop4-10}, \ref{prop4-20} and Theorem \ref{th3-40},
we obtain a formula for MZV's.
\begin{proposition}
 \label{prop4-30}
 For any multi-index $\balpha=(\alphavec)$ and any integer $k \ge 1$,
 we have
 \begin{displaymath}
  \sum_{\substack{k_1+\cdots+k_s=k\\ k_1 \ge 1,\, k_2,\,\ldots,\,k_s \ge 0}}
  \zeta_{\balpha}(\underbrace{k_1+1,1,\ldots,1}_{\alpha_1},
  \ldots,\underbrace{k_s+1,1,\ldots,1}_{\alpha_s})
  =
  \zeta \bigl(d(\balpha^{\tau}) \circledast 
  (\underbrace{1,\ldots,1}_k) \bigr).
 \end{displaymath}
\end{proposition}
We define a $\Q$-linear mapping $\zeta^{+} \colon V \to \R$ by
\begin{displaymath}
 \zeta^{+} = \zeta({}^{+}\!\bmu), \quad \bmu \in I.
\end{displaymath}
If we set $k=1$ in Proposition \ref{prop4-30}, we obtain
\begin{equation}
 \zeta_{\balpha}(\underbrace{2,1,\ldots,1}_{|\balpha|})
 = \zeta^{+} (d(\balpha^{\tau})) \label{eq4-35}
\end{equation}
for any $\balpha \in I$.
In the rest of this section, 
we prove that this assertion is nothing else but the duality for MZV's.
The duality for MZV's is the statement that for any multi-index
$\balpha$ we have
\begin{equation}
 (\ast - \tau)(\balpha) = \balpha^{*} - \balpha^{\tau} \in \ker
  \zeta^{+}. \label{eq4-30}
\end{equation}
\par
We first rewrite the left-hand side of the equation (\ref{eq4-35}).
Let $\balpha=(\alphavec)$.
By definition, we have
\begin{equation}
 \zeta_{\balpha}(\underbrace{2,1,\ldots,1}_{|\balpha|})
  = \sum_{\substack{m_1 > \cdots > n_1 \ge \cdots \ge
  m_s > \cdots > n_s > 0\\[-1mm]
  \underbrace{\hspace{3.6em}}_{\alpha_1}
  \phantom{\ge \cdots \ge \,}
  \underbrace{\hspace{3.6em}}_{\alpha_s}\phantom{\!\!>0}}}
  \,
  \frac{1}{\underbrace{m_1^2 \cdots n_1}_{\alpha_1} \cdots
  \underbrace{m_s \cdots n_s}_{\alpha_s}}. \label{eq4-40}
\end{equation}
In the following diagram, the lower and the upper arrows, respectively,
indicate the positions of the symbols $>$ and $\ge$ in the 
expression under the summation sign of (\ref{eq4-40}):
\begin{displaymath}
 {\setlength{\arraycolsep}{0pt}
 \begin{array}{cccccccccccccccccccc}
  \multicolumn{5}{c}{\overbrace{\hspace{4em}}^{\alpha_1}}&\downarrow&\multicolumn{5}{c}{\overbrace{\hspace{4em}}^{\alpha_2}}&\downarrow&
   &\downarrow&\multicolumn{5}{c}{\overbrace{\hspace{4em}}^{\alpha_s}}& \\
  \bigcirc& &\cdots& &\bigcirc& &\bigcirc& &\cdots& &\bigcirc& &\:\cdots\cdots\:& &\bigcirc& &\cdots& &\bigcirc&\quad. \\
  &\uparrow&\:\cdots\:&\uparrow& & & &\uparrow&\:\cdots\:&\uparrow& & & & & &\uparrow&\:\cdots\:&\uparrow& &
 \end{array}} 
\end{displaymath}
Let $\balpha^{*}=(\alpha_1^{*},\ldots,\alpha_t^{*})$.
Since the above diagram is the same as
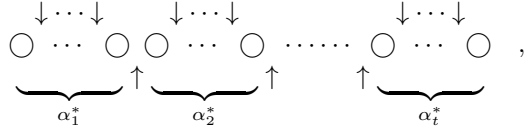
\begin{displaymath}
 {\setlength{\arraycolsep}{0pt}
 \begin{array}{cccccccccccccccccccc}
  &\downarrow&\:\cdots\:&\downarrow& & & &\downarrow&\:\cdots\:
     &\downarrow& & & & & &\downarrow&\:\cdots\:&\downarrow& & \\
  \bigcirc& &\cdots& &\bigcirc& &\bigcirc& &\cdots& &\bigcirc
   & &\:\cdots\cdots\:& &\bigcirc& &\cdots& &\bigcirc&\quad, \\
  \multicolumn{5}{c}{\underbrace{\hspace{4em}}_{\alpha^{*}_1}}&\uparrow&
   \multicolumn{5}{c}{\underbrace{\hspace{4em}}_{\alpha^{*}_2}}&\uparrow&
   &\uparrow&\multicolumn{5}{c}{\underbrace{\hspace{4em}}_{\alpha^{*}_t}}&
 \end{array}} 
\end{displaymath}
the right-hand side of (\ref{eq4-40}) is equal to
\begin{displaymath}
  \sum_{\substack{k_1 \ge \cdots \ge l_1 > \cdots >
  k_t \ge \cdots \ge l_t > 0\\[-1mm]
  \underbrace{\hspace{3.2em}}_{\alpha^{*}_1}
  \phantom{> \cdots > \!}
  \underbrace{\hspace{3.1em}}_{\alpha^{*}_t}\phantom{\!\!>0}}}
  \,
  \frac{1}{\underbrace{k_1^2 \cdots l_1}_{\alpha^{*}_1} \cdots
  \underbrace{k_t \cdots l_t}_{\alpha^{*}_t}}
  =  \zeta^{+}(u(\balpha^{*})).
\end{displaymath}
Hence, by (\ref{eq4-5}) and (\ref{eq4-7}), the equation (\ref{eq4-35})
can be rewritten as
\begin{equation}
 (\ast - \tau) d (\balpha) \in \ker \zeta^{+}. \label{eq4-50}
\end{equation}
\par
What we want to show is that the subspace $(\ast - \tau)(V)$ of 
$\ker \zeta^{+}$ spanned by the elements in (\ref{eq4-30}) is equal to
the subspace $(\ast - \tau)d(V)$ of $\ker \zeta^{+}$ spanned by
the elements in (\ref{eq4-50}).
Since the mapping $d \colon V \to V$ is a bijection 
(see~\cite[Proposition~2.3 (ii)]{K}), this assertion is clearly true.

\begin{flushleft}
 Graduate School of Mathematics \\
 Nagoya University \\
 Chikusa-ku, Nagoya 464-8602, Japan \\
 E-mail: m02009c@math.nagoya-u.ac.jp
\end{flushleft}
\end{document}